\documentclass[english,A4paper]{amsart}
\usepackage{babel}
\usepackage{amsthm,amsmath}
\usepackage{amssymb}
\usepackage[all]{xypic}
\usepackage{enumerate}
\usepackage{a4wide}
\usepackage{hyperref}
\usepackage{lscape}
\usepackage{color}

\newtheorem{theorem}{Theorem}[section]
\newtheorem{proposition}[theorem]{Proposition}

\newtheorem{lemma}[theorem]{Lemma}
\newtheorem*{theorem*}{Theorem}
\newtheorem*{proposition*}{Proposition}
\newtheorem*{corollary*}{Corollary}
\newtheorem*{lemma*}{Lemma}
\theoremstyle{definition}

\newtheorem{remark}[theorem]{Remark}
\newtheorem*{remark*}{Remark}

\newtheorem*{definition*}{Definition}
\newtheorem{aussage}[theorem]{}
\newtheorem*{aussage*}{}

\newcommand{\tensor}[1]{\otimes_{#1}}
\newcommand{\ox}{\otimes}
\def\idemp{\overline \Psi}
\def\iotaA{\alpha}
\def\iotaB{\beta}
\def\epi{\pi}
\def\mono{\iota}
\def\Wfact{\mathsf{Bf}}

\begin{document}
\title{Bilinear  factorization of algebras}
\author{Gabriella B\"ohm} 
\address{Research Institute for Particle and Nuclear Physics, Budapest,
 \newline
\indent H-1525 Budapest 114, P.O.B.\ 49, Hungary}
\email{G.Bohm@rmki.kfki.hu}
\author{Jos\'e G\'omez-Torrecillas}
\address{Departamento de \'Algebra Universidad de Granada,
           E-18071 Granada, Spain}
\email{gomezj@ugr.es}
\begin{abstract}
We study the (so-called bilinear) factorization problem answered
by a weak wreath product (of monads and, more specifically, of algebras over a
commutative ring) in the works by Street and by Caenepeel and De Groot.
A  bilinear  factorization of a monad $R$ turns out to be given by
monad morphisms $A\to R\leftarrow B$ inducing a split epimorphism of $B$-$A$
bimodules $B\ox A\to R$.
We prove a biequivalence between the bicategory of weak distributive laws and
an appropriately defined bicategory of  bilinear factorization
structures. 
As an illustration of the theory, we collect some examples of algebras over
commutative rings which admit a bilinear factorization; i.e. which
arise as weak wreath products.  
\end{abstract}
\date{Aug 2011}
\maketitle

\section*{Introduction}

A {\em distributive law} (in a bicategory) consists of two monads $A$ and $B$
together with a 2-cell $A\ox B\to B\ox A$ which is compatible with the monad
structures, see \cite{Beck}. 

A distributive law $A\ox B\to B\ox A$ is known to be equivalent to a monad
structure on the composite $B\ox A$ such that the multiplication commutes with
the actions by $B$ on the left and by $A$ on the right. The monad $B\ox A$ is
known as a {\em wreath product}, a {\em twisted product}, or a {\em smash
product} of $A$ and $B$.  

Given a monad $R$, one may ask under what conditions it is isomorphic to a
wreath product of $A$ and $B$. This question is known as a {\em (strict)
factorization} problem and the answer is this. A monad $R$ is isomorphic to a
wreath product of $A$ and $B$ if and only if there are monad morphisms $A\to R
\leftarrow B$ such that composing $B \ox A\to R\ox R$ with the multiplication
$R\ox R \to R$ yields an isomorphism $B\ox A\cong R$. In the particular
bicategory of spans this and related questions were studied in
\cite{RosWoo}. In the monoidal category (i.e. one object bicategory) of
modules over a commutative ring; and also in its opposite, such questions were
investigated in \cite{CIMZ}, see also \cite{Majid:1990} and
\cite{Tambara:1990}.  

In the papers \cite{Stef&Erwin} and \cite{Street:2009}, the notion of
distributive law was generalized by weakening the compatibility conditions with
the units of the monads. A so defined {\em weak distributive law} $A\ox B\to
B\ox A$ also induces an associative multiplication on $B\ox A$ but it fails to
be unital. However, there is a canonical idempotent on $B\ox A$. Whenever it
splits, the corresponding retract is a monad, known as the {\em weak wreath
product} or {\em weak smash product} of $A$ and $B$, see \cite{Street:2009}
and \cite{Stef&Erwin}. 

The aim of this paper is to study the factorization problem answered by a weak
wreath product. 
In the particular bicategory of spans, this problem has already been studied
in \cite{Bohm_wfactor}. 
In the paper \cite{VilRodRap} addressing questions of similar motivation, a
more general notion of {\em weak crossed product monad} was considered. Such
weak crossed products are not induced by weak distributive laws but by more
general 1-cells in an extended bicategory of monads introduced in
\cite{Boh_wtm}. The factorization problem corresponding to weak crossed
products is fully described in \cite{VilRodRap}. 

We start Section \ref{sec:wwp_vs_wfac} by recalling from \cite{Street:2009}
the notion of weak distributive law and the corresponding construction of weak
wreath product. We show that a monad $R$ is isomorphic to a weak wreath
product of some monads $A$ and $B$ if and only if there are monad morphisms
(with trivial 1-cell parts) $A\to R \leftarrow B$ such that composing $B\ox A
\to R\ox R$ with the multiplication $R\ox R\to R$ yields a split epimorphism
of $B$-$A$ bimodules $B\ox A \to R$. What is more, in Theorem
\ref{thm:bieq}, for any bicategory in which idempotent 2-cells split, we prove
a biequivalence of the bicategory of weak distributive laws and an
appropriately defined bicategory of  bilinear factorization
structures. This extends \cite[Theorem 3.12]{Bohm_wfactor}.

Section \ref{sec:examples} is devoted to collecting examples of algebras over
commutative rings which admit a  bilinear  factorization. 

The algebra homomorphisms $A\to R\leftarrow B$ in a bilinear
factorization structure are not injective in general. In Paragraph
\ref{as:injective} we show, however, that if $R$ admits any  bilinear
factorization then it admits also a bilinear  factorization
with injective algebra homomorphisms $\tilde A\to R\leftarrow \tilde B$. In
general the latter factorization is still non-strict and we
characterize those cases when it happens to be strict.   

In Paragraph \ref{as:e} we consider an algebra $A$ and an element $e$ of it
such that $ea=eae$ for all $a\in A$ (so that $eA$ is an algebra with unit
$e$). Assuming that there is a strict distributive law $eA\ox B \to B\ox eA$,
we extend it to a weak distributive law $A\ox B \to B\ox A$. The corresponding
weak wreath product is isomorphic to the strict wreath product of $eA$ and
$B$; hence it admits a strict factorization in terms of them. 

The Ore extension of an algebra $B$ over a commutative ring $k$ is the wreath
product of $B$ with the algebra $k[X]$ of polynomials of a formal variable
$X$, see \cite[Example 2.11 (1)]{CIMZ}. In Paragraph \ref{as:Ore},
generalizing Ore extensions, we construct a weak wreath product of $B$ with
$k[X]$, that we regard as a weak Ore extension of $B$ (although it turns out
to be isomorphic in a nontrivial way to a strict Ore extension of an
appropriate subalgebra $\tilde B$). 

For any commutative ring $k$, there is a bicategory $\mathsf{Bim}$ of
$k$-algebras, their bimodules and bimodule maps. In Paragraph \ref{as:sF} we
consider strict distributive laws in $\mathsf{Bim}$. Taking a 0-cell
(i.e. $k$-algebra) $R$ which admits a separable Frobenius structure, we show
that any distributive law over $R$ induces a weak distributive law over
$k$. The corresponding weak wreath product is isomorphic to the $R$-module
tensor product. The examples in Paragraph \ref{as:dirsum} and Paragraph
\ref{as:wba} belong to this class of examples.

In Paragraph \ref{as:dirsum} we start with a finite collection of strict
distributive laws $A_i\ox B_i \to B_i \ox A_i$ and construct a weak
distributive law $(\oplus_i A_i)\ox (\oplus_i B_i)\to (\oplus_i B_i) \ox
(\oplus_i A_i)$. The corresponding weak wreath product is isomorphic to the
direct sum of the wreath product algebras $B_i\ox A_i$. 

In Paragraph \ref{as:wba} we take a weak bialgebra $H$ and an $H$-module
algebra $A$. We show that their smash product is a weak wreath product.

In Paragraph \ref{as:2x2=3} we present explicitly a bilinear 
factorization of the three dimensional noncommutative algebra of $2 \times 2$
upper triangle matrices with entries in a field $k$ whose characteristic is
different from $2$, in terms of two copies of the commutative algebra $k\oplus
k$.  
\bigskip

{\bf Acknowledgements.} GB thanks all members of Departamento de \'Algebra
at Universidad de Granada for a generous invitation and for a very warm
hospitality experienced during her visit in June 2010 when the work on this
paper begun. Partial financial support from the Hungarian Scientific Research
Fund OTKA, grant no. K68195, and from the Ministerio de Ciencia e Innovaci\'on
and FEDER, grant MTM2010-20940-C02-01 is gratefully acknowledged.

\section{Weak wreath products and bilinear factorizations}
\label{wpalgebras}
\label{sec:wwp_vs_wfac}

In this section we work in a bicategory $\mathcal K$, whose coherence
isomorphisms will be omitted in our notation. The horizontal composition is
denoted by $\ox$ and the vertical composition is denoted by juxtaposition. 
Our motivating example is the one-object bicategory (i.e. monoidal category) of
modules over a commutative ring (where $\ox$ is the module tensor product). 

\begin{aussage} {\bf Weak distributive laws.}
Let $(A,\mu_A,\eta_A)$ and $(B,\mu_B,\eta_B)$ be (associative and unital)
monads in $\mathcal K$ on the same object, with multiplications $\mu_A, \mu_B$
and units $\eta_A, \eta_B$. 
Following \cite[Theorem 3.2]{Stef&Erwin} and \cite[Definition
2.1]{Street:2009}, a 2-cell $\Psi : A \tensor{} B \rightarrow B \tensor{}
A$ is said to be a \emph {weak distributive law of $A$ over $B$} if the
following diagrams commute. 
\begin{equation}\label{eq:wdl}
\xymatrix @R=15pt{
A\ox A\ox B \ar[r]^-{\mu\ox B}\ar[d]_-{A\ox \Psi}&
A\ox B\ar[dd]^-\Psi
&
A\ox B\ox B\ar[r]^-{A\ox \mu}\ar[d]_-{\Psi\ox B}&
A\ox B \ar[dd]^-\Psi &
B\ox A \ar[r]^-{B\ox A \ox \eta}\ar[d]_-{\eta \ox B\ox A}&
B\ox A \ox B \ar[d]^-{B\ox \Psi}\\
A\ox B\ox A\ar[d]_-{\Psi\ox A}&
&
B\ox A\ox B \ar[d]_-{B\ox \Psi}&
&
A\ox B \ox A \ar[d]_-{\Psi\ox A}&
B\ox B \ox A \ar[d]^-{\mu \ox A}\\
B\ox A\ox A\ar[r]_-{B\ox \mu}&
B\ox A 
&
B\ox B \ox A\ar[r]_-{\mu\ox A}&
B\ox A 
&
B\ox A \ox A\ar[r]_-{B\ox \mu}&
B\ox A
}
\end{equation}
\end{aussage}

\begin{lemma}\cite[Proposition 2.2]{Street:2009}
The third diagram in \eqref{eq:wdl} is equivalent to the following two
diagrams.
\begin{equation}\label{eq:wdl_alt}
\xymatrix @R=15pt{
B\ar[r]^-{\eta\ox B}\ar[d]_-{B\ox \eta\ox \eta}&
A \ox B \ar[dd]^-{\Psi}
&
A \ar[r]^-{A\ox \eta}\ar[d]_-{\eta\ox \eta \ox A}&
A\ox B\ar[dd]^-{\Psi}\\
B\ox A \ox B \ar[d]_-{B\ox \Psi}&
&
A\ox B \ox A \ar[d]_-{\Psi\ox A}&
&\\
B\ox B \ox A \ar[r]_-{\mu \ox A}&
B\ox A 
&
B\ox A \ox A \ar[r]_-{B\ox \mu}&
B\ox A}
\end{equation}
\end{lemma}
\begin{proof}
The following diagram shows that commutativity of the third diagram in
\eqref{eq:wdl} implies commutativity of the first diagram in
\eqref{eq:wdl_alt},
$$
\xymatrix @R=15pt {
B\ar[r]^-{B\ox \eta}\ar[d]_-{\eta\ox B}&
B\ox A \ar[r]^-{B\ox A \ox \eta}\ar[d]^-{\eta\ox B\ox A}&
B\ox A \ox B \ar[r]^-{B\ox \Psi}&
B\ox B \ox A\ar[dd]^-{\mu \ox A}\\
A\ox B \ar[r]^-{A\ox B \ox \eta}\ar[d]_-{\Psi}&
A\ox B \ox A \ar[d]^-{\Psi\ox A}&&\\
B\ox A \ar[r]^-{B\ox A \ox \eta}\ar@{=}@/_1.5pc/[rrr]&
B\ox A \ox A \ar[rr]^-{B\ox \mu}&&
B\ox A}
$$
\vspace{.6cm}

\noindent
where the region on the right commutes by the third diagram in
\eqref{eq:wdl}. Commutativity of the second diagram in \eqref{eq:wdl_alt} 
is verified symmetrically. 
Conversely, if both diagrams in \eqref{eq:wdl_alt} commute then so does
\begin{equation}\label{eq:kappas}
\xymatrix @R=12pt{
B\ox A \ar[rrr]^-{\eta\ox B \ox A}
\ar[ddd]_-{B\ox A \ox \eta}\ar[rd]^(.6){B\ox \eta \ox \eta \ox A}&&&
A\ox B \ox A \ar[dd]^-{\Psi\ox A}\\
&B\ox A \ox B \ox A \ar[rd]^-{B\ox \Psi \ox A}&&\\
&&B\ox B \ox A \ox A \ar[r]^-{\mu \ox A \ox A}\ar[d]_-{B\ox B \ox \mu}
\ar[rd]^-{\mu\ox \mu}&
B\ox A \ox A \ar[d]^-{B\ox \mu}\\
B\ox A \ox B \ar[rr]_-{B\ox \Psi}&&
B\ox B \ox A \ar[r]_-{\mu \ox A}&
B\ox A}
\end{equation}
\end{proof}

\begin{aussage} {\bf Weak wreath product.}
Define $\mu : B \tensor{} A \tensor{} B \tensor{} A
\rightarrow B \tensor{} A$ as 
$$
\xymatrix{
B\ox A \ox B \ox A \ar[rr]^-{B\ox \Psi \ox A}&&
B\ox B \ox A \ox A \ar[r]^-{\mu\ox \mu}&
B\ox A.}
$$
It follows from the first two diagrams in \eqref{eq:wdl} that $\mu$  is an
associative multiplication. 
From now on, we consider $B \tensor{} A$ as an associative monad with the
multiplication $\mu$ -- possibly without a unit. (In fact,
$B\ox A$ can be seen to possess a preunit $\eta_B \ox \eta_A$ in the sense
discussed in \cite{Stef&Erwin}.)  
\end{aussage}

\begin{proposition} \label{prop:kappa}
(See \cite[Proposition 2.3]{Street:2009}.)
For any weak distributive law $\Psi:A\ox B \to B \ox A$, define $\idemp : B
\tensor{} A \rightarrow B \tensor{} A$ by  
\begin{equation}\label{wdl6}
\xymatrix{
B\ox A \ar[rr]^-{B \ox \eta_A \ox \eta_B \ox A}&&
B\ox A \ox B \ox A \ar[r]^-\mu&
B\ox A.}
\end{equation}
Then $\idemp$ is an idempotent endomorphism of monads (without unit), and of
$B$-$A$ bimodules. Moreover, 
$\idemp \Psi = \Psi$.
\end{proposition}
\begin{proof}
Note that $\idemp$ stands in the diagonal of the diagram
\eqref{eq:kappas}. Hence it has the equivalent forms
\begin{eqnarray}
\idemp&=&(B\ox \mu_A)(\Psi\ox A)(\eta_A\ox B \ox A)\label{eq:kappaL}\\
&=&(\mu_B \ox A)(B\ox \Psi)(B\ox A \ox \eta_B) \label{eq:kappaR}.
\end{eqnarray}
Since the expression \eqref{eq:kappaL} is evidently a right $A$-module map and
\eqref{eq:kappaR} is a left $B$-module map, this proves the bilinearity of
$\idemp$, i.e. 
\begin{equation}\label{wdl8}
(\mu_B \tensor{} A)(B \tensor{} \idemp) = \idemp(\mu_B \tensor{} A) 
\qquad \textrm{and}\qquad 
(B \tensor{} \mu_A)(\idemp \tensor{} A) = \idemp(B \tensor{} \mu_A).
\end{equation}
By commutativity of
$$
\xymatrix {
A\ox B  \ar[r]^-{\Psi}\ar[d]^-{\eta\ox A \ox B}
\ar@{=}@/_3pc/[ddd]&
B\ox A \ar[d]^-{\eta\ox B\ox A}\\
A\ox A \ox B \ar[r]^-{A\ox \Psi}\ar[dd]^(.25){\mu\ox B}
\ar@{}[rdd]|{\eqref{eq:wdl}}&
A\ox B \ox A \ar[d]^-{\Psi\ox A}\\
&B\ox A \ox A\ar[d]^-{B\ox \mu}\\
A\ox B \ar[r]_-{\Psi }&
B\ox A}
$$
and \eqref{eq:kappaL}, we obtain $\idemp\Psi =\Psi$. This implies 
\begin{equation}\label{wdl12}
\idemp \mu =
\idemp(\mu_B \ox  \mu_A)(B\ox \Psi\ox A)
\stackrel{\eqref{wdl8}}{=}
(\mu_B\ox \mu_A)(B\ox \idemp\ox A)(B\ox \Psi\ox A)= 
\mu,
\end{equation}
hence also $\idemp^2=\idemp$. Moreover, by commutativity of
$$
\xymatrix{
A\ox B \ox A \ar[rr]^-{A\ox \eta \ox B\ox A}\ar@{=}[rrdd]&&
A\ox A \ox B\ox A \ar[r]^-{A\ox \Psi\ox A}\ar[dd]^-{\mu\ox B \ox A}
\ar@{}[rdd]|{\eqref{eq:wdl}}&
A\ox B \ox A\ox A \ar[r]^-{A\ox B \ox \mu} \ar[d]^-{\Psi\ox A \ox A}&
A\ox B \ox A\ar[d]^-{\Psi \ox A}\\
&&&B\ox A \ox A \ox A\ar[r]^-{B\ox A \ox \mu}\ar[d]^-{B\ox \mu \ox A}&
B\ox A \ox A \ar[d]^-{B\ox \mu}\\
&&A\ox B \ox A \ar[r]_-{\Psi\ox A}&
B\ox A \ox A \ar[r]_-{B\ox \mu}&
B\ox A}
$$
and \eqref{eq:kappaL}, we obtain $(B\ox \mu_A)(\Psi\ox A)(A\ox \idemp)=(B\ox
\mu_A)(\Psi\ox A)$.  
This implies that 
$$
\mu(B\ox A \ox \idemp)=
(\mu_B \ox A)(B\ox B \ox \mu_A)(B\ox \Psi \ox A)(B\ox A \ox \idemp)=
\mu.
$$
Combining it with the symmetrical counterpart, we conclude that
\begin{equation}\label{wdl11}
\mu(\idemp \tensor{} \idemp) = \mu.
\end{equation}
From \eqref{wdl11} and \eqref{wdl12} we get that $\idemp$ is multiplicative
with respect to $\mu$.
\end{proof}

\begin{aussage}{\bf Splitting idempotents.}
Assume that the idempotent 2-cell $\idemp$ associated in Proposition
\ref{prop:kappa} to a weak distributive law $\Psi$ splits. That is, there is a
(unique up-to isomorphism) 1-cell $B\otimes_\Psi A$ and 2-cells $\pi:B\ox A
\to B\otimes_\Psi A$ and $\iota:B\otimes_\Psi A \to B\ox A$ such that
$\pi\iota=B\otimes_\Psi A$ and $\iota\pi=\idemp$.  
Since $\idemp$ is a morphism of $B$-$A$ bimodules, there is a unique $B$-$A$
bimodule structure on $B\ox_\Psi A$ such that both $\pi$ and $\iota$ are
morphisms of $B$-$A$ bimodules (i.e. $B\ox_\Psi A$ is a $B$-$A$ bimodule
retract of $B\ox A$).
\end{aussage}

\begin{theorem}\label{BPsiA} (See \cite[Theorem 2.4]{Street:2009}.)
Let $\Psi : A \tensor{} B \rightarrow B \tensor{} A$ be a weak distributive
law  in a bicategory $\mathcal K$, such that the associated idempotent 2-cell
$\idemp$ splits. Then there is a retract monad $(B \tensor{\Psi} A,\mu_\Psi)$
of $(B \tensor{} A, \mu)$ which is unital. Moreover, the 2-cells
$$
\iotaB := \pi(B \tensor{} \eta_A) : B \rightarrow B \tensor \Psi A, \qquad
\iotaA := \pi(\eta_B \tensor{} A) : A \rightarrow B \tensor \Psi A 
$$
are homomorphisms of unital monads such that $\mu_\Psi(\iotaB \tensor{}
\iotaA) : B \tensor{} A \rightarrow B \tensor{\Psi} A$ is  equal to $\pi$; and
the left $B$- and right $A$-actions on $B\ox_\Psi A$ can be written as
$\mu_\Psi(\iotaB\ox (B\ox_\Psi A))$ and $\mu_\Psi((B\ox_\Psi A)\ox \iotaA)$,
respectively. 
\end{theorem}
\begin{proof}
Equip $B\ox_\Psi A$ with the multiplication
$$
\mu_\Psi:=\big(
\xymatrix{
(B\ox_\Psi A)\ox (B\ox_\Psi A) \ar[r]^-{\iota \ox \iota}&
B \ox A \ox B \ox A \ar[r]^-\mu&
B\ox A \ar[r]^-\pi&
B\ox_\Psi A}\big).
$$
By \eqref{wdl11}, $\pi\mu=\mu_\Psi(\pi\ox \pi)$ and by \eqref{wdl12},
$\mu(\iota\ox \iota)=\iota \mu_\Psi$. Since $\iota$ is a (split) monomorphism
and $\pi$ is a (split) epimorphism, any of these equalities implies
associativity of $\mu_\Psi$. It is also unital with $\eta_\Psi:=\pi(\eta_B \ox
\eta_A)$ since 
$$
\mu_\Psi((B\ox_\Psi A)\ox \pi)((B\ox_\Psi A)\ox \eta_B\ox \eta_A)\pi
\stackrel{\eqref{wdl11}}{=}
\pi\mu(B\ox A\ox \eta_B\ox \eta_A)
\stackrel{\eqref{eq:kappaR}}{=}
\pi\idemp=
\pi,
$$
and symmetrically on the other side.
Unitality of $\iotaB$ is evident. 
We have $\iota\beta\mu_B=\idemp (B\ox \eta_A)\mu_B$ and
by \eqref{wdl12} and \eqref{wdl11}, $\iota \mu_\Psi(\iotaB \ox
\iotaB)=\mu(B\ox \eta_A \ox B\ox \eta_A)$. Hence multiplicativity of $\iotaB$
follows by commutativity of   
$$
\xymatrix{
B\ox B \ar[rr]^-{B\ox \eta\ox B}\ar[dd]_-{\mu}\ar[rd]^-{B\ox B \ox \eta}&&
B\ox A \ox B \ar[rr]^-{B\ox A \ox B\ox \eta}\ar[d]^-{B\ox \Psi}
\ar@{}[ld]|{\eqref{eq:kappaL}}&&
B\ox A \ox B \ox A \ar[d]^-{B\ox \Psi\ox A}\\
&B\ox B \ox A\ar[r]^-{B\ox \idemp}\ar[d]^-{\mu\ox A}&
B\ox B \ox A \ar[rr]^-{B\ox B \ox A \ox \eta}\ar[rrd]_(.4){\mu\ox A}
\ar@{}[d]|{\eqref{wdl8}}&&
B\ox B \ox A \ox A \ar[d]^-{\mu\ox \mu}\\
B\ar[r]_-{B\ox \eta}&
B\ox A \ar[rrr]_-\idemp&&&
B\ox A\ .
}
$$
That $\iotaA$ is an algebra homomorphism follows by symmetry. Finally, 
$$
\iota \mu_\Psi(\beta\ox (B\ox_\Psi A))(B\ox \pi)
\stackrel{\eqref{wdl12}\eqref{wdl11}}=
\mu(B\ox \eta_A \ox B \ox A)
\stackrel{\eqref{eq:kappaL}}=
(\mu_B\ox A)(B\ox \idemp)=
(\mu_B\ox A)(B\ox \iota\pi)
$$
so that $\mu_\Psi(\beta\ox (B\ox_\Psi A))=\pi(\mu_B\ox A)(B\ox \iota)$ as
stated, and symmetrically for the right $A$-action. Therefore,
$$
\mu_\Psi(\beta\ox \alpha)=
\pi(\mu_B\ox A)(B\ox \idemp)(B\ox \eta_B \ox A)
\stackrel{\eqref{wdl8}}=
\pi\idemp=
\pi.
$$ 
\end{proof}

The situation in the above theorem motivates the following notion.

\begin{aussage} {\bf  Bilinear factorization structures.}
In an arbitrary bicategory $\mathcal K$, consider unital monads $(A, \mu_A,
\eta_A)$, $(B, \mu_B,\eta_B)$ and $ (R, \mu_R,\eta_R)$. Let $\iotaA: A
\rightarrow R \leftarrow B:\iotaB$ be 2-cells which are compatible with the
monad structures in the sense of the diagrams 
$$
\xymatrix{
A\ox A \ar[r]^-{\iotaA\ox \iotaA}\ar[d]_-\mu&
R\ox R \ar[d]^-\mu&
B\ox B \ar[l]_-{\iotaB\ox \iotaB}\ar[d]^-\mu&
k\ar@{=}[r]\ar[d]_-\eta&
k\ar[d]^-\eta&
k\ar@{=}[l]\ar[d]^-\eta\\
A\ar[r]_-\iotaA&
R&
B\ar[l]^-\iotaB&
A\ar[r]_-\iotaA&
R&
B\ar[l]^-\iotaB \ ;}
$$
i.e. $\iotaA$ and $\iotaB$ be morphisms of (unital) monads. (They are
monad morphisms with trivial 1-cell parts in the sense of \cite{Street}.)
Regarding $R$ as a left $B$-module via $\mu_R (\iotaB \tensor{} R):B\ox R
\to R$ and a right $A$-module via $\mu_R(R \tensor{} \iotaA):R\ox A \to R$, 
\begin{equation}\label{nu}
\epi:= 
\big(
\xymatrix{
B\ox A \ar[r]^-{\iotaB \tensor{} \iotaA}&
R\ox R \ar[r]^-\mu&
R}\big)
\end{equation}
is a homomorphism of $B$-$A$ bimodules. 
If $\epi$ has a $B$-$A$ bimodule section $\mono$, then we call the datum
$(\iotaA:A\to R\leftarrow B:\beta,\mono:R\to B\ox A)$ a {\em bilinear
factorization structure on $R$} or, shortly, a {\em bilinear factorization
of $R$}.  
\end{aussage}

By Theorem \ref{BPsiA}, any weak distributive law $\Psi:A\ox B \to B \ox A$
for which the idempotent 2-cell $\idemp$ splits, determines a bilinear 
factorization structure $(\iotaA:A\to B\ox_\psi A\leftarrow
B:\iotaB,\iota:B\ox_\Psi A\to B\ox A)$. We turn to proving the converse. 

\begin{theorem}\label{thm:weak_fact}
For a  bilinear factorization structure  $(\iotaA:A\to R\leftarrow
B:\beta,\mono:R\to B\ox A)$ in an arbitrary bicategory $\mathcal K$, 
$$
\Psi:= \big(
\xymatrix{
A\ox B \ar[r]^-{\iotaA \tensor{} \iotaB}&
R\ox R \ar[r]^-\mu&
R \ar[r]^-{\mono}&
B\ox A}\big)
$$ 
is a weak distributive law of  $A$ over $B$ such that the corresponding
idempotent 2-cell $\idemp$ splits. Moreover, $R$ is isomorphic to the
corresponding unital monad $B \tensor{\Psi} A$.    
\end{theorem}
\begin{proof}
The assumption that $\mono$ is a morphism of $B$-$A$ bimodules means the
equalities 
\begin{equation}\label{iotaAmod}
\mono \mu_R (R \tensor{} \iotaA) = (B \tensor{} \mu_A)(\mono \tensor{} A)
\qquad \textrm{and}\qquad
\mono \mu_R(\iotaB \tensor{} R) = (\mu_B \ox A)(B \tensor{} \mono). 
\end{equation}
Compatibility of $\Psi$ with the multiplication of $A$ (i.e. the first
diagram in
\eqref{eq:wdl}) follows by commutativity of 
$$
\xymatrix{
A\ox A \ox B\ar[rrr]^-{\mu\ox B}\ar[d]_-{A\ox\iotaA \ox \iotaB}&&&
A\ox B \ar[d]^-{\iotaA\ox \iotaB}\\
A\ox R \ox R\ar[rr]^-{\iotaA\ox R \ox R}\ar[d]_-{A\ox \mu}&&
R\ox R \ox R \ar[r]^-{\mu\ox R}\ar@/^2pc/[ddd]^-{R\ox \mu}&
R\ox R\ar[ddd]^-\mu\\
A\ox R \ar[d]_-{A\ox \mono}\ar@{=}[rrd]\ar@{}[rd]_-{(*)}\\
A\ox B \ox A\ar[r]_-{A\ox \iotaB \ox \iotaA}\ar[d]_-{\iotaA\ox \iotaB\ox A}&
A\ox R \ox R\ar[r]_-{A\ox \mu}\ar[d]^-{\iotaA\ox R \ox R}&
A\ox R \ar[d]_-{\iotaA \ox R}\\
R\ox R\ox A\ar[r]^-{R\ox R\ox \iotaA}\ar[d]_-{\mu\ox A}&
R\ox R\ox R\ar[r]^-{R\ox \mu}\ar[d]^-{\mu\ox R}&
R\ox R\ar[r]^-\mu&
R\ar[dd]^-{\mono}\\
R\ox A \ar[r]^-{R\ox \iotaA}\ar[d]_-{\mono\ox A}
\ar@{}[rrrd]|-{\eqref{iotaAmod}}&
R\ox R\ar[rru]_-{\mu}\\
B\ox A\ox A\ar[rrr]_-{B\ox \mu}&&&
B\ox A.
}
$$
The top region commutes by the multiplicativity of $\alpha$ and the region
labelled by ($\ast$) commutes since $\mono$ is a section of $\epi$ (occurring
at the bottom of this region).  
It follows by symmetrical considerations that $\Psi$ renders commutative also
the second diagram in \eqref{eq:wdl}.
As for the third one concerns, in the diagram
$$
\xymatrix{
B\ox A \ar[r]^-{\eta\ox B\ox A}\ar[d]_-{\iotaB\ox \iotaA}&
A\ox B \ox A \ar[r]^-{\iotaA\ox \iotaB \ox A}&
R\ox R\ox A\ar[r]^-{\mu\ox A}\ar[d]_-{R\ox R\ox \iotaA}&
R\ox A \ar[rr]^-{\mono\ox A}\ar[d]^-{R\ox \iotaA}
\ar@{}[rrd]|-{\eqref{iotaAmod}}
&&
B\ox A\ox A\ar[d]^-{B\ox \mu}\\
R\ox R\ar[rr]^-{\eta\ox R\ox R}\ar@{=}@/_1pc/[rrr]&&
R\ox R\ox R\ar[r]^-{\mu\ox R}&
R\ox R\ar[r]^-\mu&
R\ar[r]^-{\mono}&
B\ox A}
$$
\vspace{.2cm}

\noindent
the region on the left commutes by the unitality of $\iotaA$. 
Commutativity of this diagram yields the equality
\begin{equation}\label{kappaiotanu}
(B \tensor{} \mu_A)(\Psi \tensor{} A)(\eta_A \tensor{} B \tensor{} A) = 
\mono \epi.
\end{equation}  
Symmetrically, 
\[
(\mu_B \tensor{} A)(B \tensor{} \Psi)(B \tensor{} A \tensor{} \eta_B) = 
\mono \epi 
\]
which proves that $\Psi$ renders commutative the third diagram in
\eqref{eq:wdl}, so that $\Psi$ is a weak distributive law.   

By \eqref{eq:kappaL}, the expression on the left hand side of
\eqref{kappaiotanu} is $\idemp$ which clearly splits. The corresponding 1-cell
$B\ox_\Psi A$ is defined (uniquely up-to isomorphism) via some splitting of it
as $\pi_\Psi:B\ox A \to B\ox_\Psi A$ and $\iota_\Psi: B\ox_\Psi A\to B\ox A$. 
By uniqueness up-to isomorphism of the splitting of an idempotent
2-cell, \eqref{kappaiotanu} implies that $B\ox_\Psi A$ and $R$ are isomorphic
1-cells in $\mathcal K$ via the mutually inverse isomorphisms
$\pi_\Psi\mono:R\to B\ox_\Psi A$ and $\epi\iota_\Psi:B\ox_\Psi A \to R$. 

Composing both equal paths in
$$
\xymatrix{
B\ox R \ox R \ox A\ar[rr]^-{\iotaB\ox R \ox R\ox \iotaA}
\ar[d]_-{B\ox \mu \ox A}&&
R\ox R\ox R \ox R\ar[r]^-{\mu\ox \mu}\ar[d]^-{R\ox \mu\ox R}&
R\ox R\ar[d]^-\mu\\
B\ox R\ox A\ar[rr]^-{\iotaB\ox R\ox \iotaA}\ar[d]_-{B\ox \mono\ox A}
\ar@{}[rrrd]|-{\eqref{iotaAmod}}&&
R\ox R\ox R\ar[r]^-{\mu^2}&
R\ar[d]^-{\mono}\\
B\ox B\ox A\ox A\ar[rrr]_-{\mu\ox \mu}&&&
B\ox A,
}
$$
by $B\ox \iotaA\ox \iotaB\ox A$ on the right, we obtain 
\begin{equation}\label{eq:ip_multi}
\mono \mu_R (\epi\ox \epi) = 
(\mu_B \tensor{} \mu_A)(B \tensor{} \Psi \tensor{} A), 
\end{equation}
hence multiplicativity of $\epi$ (and $\mono$). Since $\iota_\Psi$ is
multiplicative by \eqref{wdl12}, so is $\epi\iota_\Psi$.
Finally, 
$$
\epi\iota_\Psi\eta_\Psi=
\epi\iota_\Psi\pi_\Psi(\eta_B\ox \eta_A)
\stackrel{\eqref{kappaiotanu}}=
\epi(\eta_B\ox \eta_A)
\stackrel{\eqref{nu}}=
\mu_R(\eta_R\ox \eta_R)=
\eta_R.
$$
\end{proof}

We close this section by proving that the constructions in Theorem \ref{BPsiA}
and Theorem \ref{thm:weak_fact} can be regarded as the object maps of a
biequivalence between appropriately defined bicategories. 

The bicategory of {\em mixed} weak distributive laws was studied in
\cite{BoLaSt_wdlaw}. Taking the dual notion, we obtain the following.

\begin{aussage} {\bf The bicategory of weak distributive laws.}
The 0-cells of the bicategory $\mathsf{Wdl}({\mathcal K})$ are weak
distributive laws $\Psi:A\ox B \to B \ox A$ in the bicategory $\mathcal
K$. The 1-cells between them consist of monad morphisms (in the sense of
\cite{Street}) $\xi: A'\ox V\to V\ox A$ and $\zeta:B'\ox V \to V\ox B$ with a
common 1-cell $V$ such that the following diagram commutes.
\begin{equation}\label{eq:wdl_1-cell}
\xymatrix{
A'\ox B'\ox V\ar[r]^-{A'\ox \zeta}\ar[d]_-{\Psi'\ox V}&
A'\ox V \ox B\ar[r]^-{\xi\ox B}&
V\ox A \ox B \ar@{=}[r]&
V\ox A \ox B\ar[d]^-{V\ox \Psi}\\
B'\ox A'\ox V\ar[r]_-{B'\ox \xi}&
B'\ox V\ox A \ar[r]_-{\zeta\ox A}&
V\ox B \ox A\ar[r]_-{V\ox \idemp}&
V\ox B \ox A}
\end{equation}
The 2-cells are those 2-cells $\omega:V\to V'$ in $\mathcal K$ which are monad
transformations (in the sense of \cite{Street}) $(V,\xi)\to (V',\xi')$ and
$(V,\zeta)\to (V',\zeta')$. Horizontal and vertical compositions are induced
by those in $\mathcal K$. 
\end{aussage}

\begin{aussage} {\bf The bicategory of bilinear factorization
structures.}
The 0-cells of the bicategory $\Wfact({\mathcal K})$ are the bilinear 
factorization structures  $(\iotaA:A\to R\leftarrow B:\iotaB,\mono:R\to
B\ox A)$ in the bicategory $\mathcal K$. The 1-cells between them are triples
of monad morphisms (in the sense of \cite{Street}) $\xi: A'\ox V\to V\ox A$,
$\zeta:B'\ox V \to V\ox B$ and $\varrho:R'\ox V \to V \ox R$ with a 
common 1-cell $V$ such that the following diagrams commute.
\begin{equation}  \label{eq:wfac_1-cell}
\xymatrix{
A'\ox V \ar[r]^-{\iotaA'\ox V}\ar[d]_-\xi&
R'\ox V \ar[d]^-{\varrho}&&
B'\ox V \ar[r]^-{\iotaB'\ox V}\ar[d]_-\zeta&
R'\ox V \ar[d]^-{\varrho}\\
V\ox A \ar[r]_-{V\ox \iotaA}&
V\ox R &&
V\ox B \ar[r]_-{V\ox \iotaB}&
V\ox R}
\end{equation}
The 2-cells are those 2-cells $\omega:V\to V'$ in $\mathcal K$ which are monad
transformations (in the sense of \cite{Street}) $(V,\xi)\to (V',\xi')$,
$(V,\zeta)\to (V',\zeta')$ and $(V,\varrho)\to (V',\varrho')$. Horizontal and
vertical compositions are induced by those in $\mathcal K$. 
\end{aussage}

\begin{aussage} \label{as:F}
{\bf A pseudofunctor 
$F:\Wfact({\mathcal K})\to \mathsf{Wdl}({\mathcal K})$.}  
The pseudofunctor $F$ takes a bilinear factorization structure 
$(\iotaA:A\to R\leftarrow B:\iotaB,\mono:R\to B\ox A)$ to the corresponding
weak distributive law $\Psi:= \mono\mu_R(\iotaA \tensor{} \iotaB): A \tensor{}
B \rightarrow B \tensor{} A$ in Theorem \ref{thm:weak_fact}. 
It takes a 1-cell $(\xi,\zeta,\varrho)$ to $(\xi,\zeta)$. 
On the 2-cells $F$ acts as the identity map.

The only non-trivial point to see is that $(\xi,\zeta)$ is indeed a 1-cell in 
$\mathsf{Wdl}({\mathcal K})$ by commutativity of the following diagram. 
$$
\xymatrix @R=15pt {
A'\ox B'\ox V \ar[r]^-{A'\ox \zeta}\ar[d]_-{\iotaA'\ox \iotaB'\ox V}
\ar@{}[rd]|-{\eqref{eq:wfac_1-cell}}&
A'\ox V '\ox B\ar[rr]^-{\xi\ox B}\ar[d]^-{\iotaA'\ox V\ox \iotaB}
\ar@{}[rrd]|-{\eqref{eq:wfac_1-cell}}&&
V\ox A \ox B \ar[d]^-{V\ox \iotaA \ox \iotaB}\\
R'\ox R'\ox V\ar[r]_-{R'\ox \varrho}\ar[d]_-{\mu'\ox V}&
R'\ox V \ox R\ar[rr]_-{\varrho\ox R}&&
V\ox R \ox R \ar[d]^-{V\ox \mu}\\
R'\ox V \ar[rrr]^-\varrho\ar[d]_-{\mono'\ox V}&&&
V\ox R \ar[d]^-{V\ox \mono}\\
B'\ox A'\ox V \ar[r]_-{B'\ox \xi}&
B'\ox V \ox A \ar[r]_-{\zeta \ox A}&
V\ox B\ox A \ar[ru]^-{V\ox \epi}_-{\eqref{kappaiotanu}}
\ar[r]_-{V\ox \idemp}&
V\ox B \ox A}
$$
The middle region commutes since $\varrho$ is a monad morphism. 
The bottom region commutes by commutativity of
$$
\xymatrix @C=35pt @R=15pt {
B'\ox A'\ox V\ar[r]^-{\iotaB'\ox \iotaA'\ox V}\ar[d]_-{B'\ox \xi}
\ar@{}[rd]|-{\eqref{eq:wfac_1-cell}}&
R'\ox R'\ox V\ar[r]^-{\mu'\ox V}\ar[d]^-{R'\ox \varrho}&
R'\ox V \ar[dd]^-\varrho\\
B'\ox V \ox A \ar[r]^-{\iotaB'\ox V\ox \iotaA}\ar[d]_-{\zeta\ox A}
\ar@{}[rd]|-{\eqref{eq:wfac_1-cell}}&
R'\ox V \ox R\ar[d]^-{\varrho \ox R}\\
V\ox B \ox A \ar[r]_-{V\ox \iotaB\ox \iotaA}&
V\ox R \ox R\ar[r]_-{V\ox \mu}&
V\ox R}
$$
which, in light of \eqref{nu}, means $(V\ox \epi)(\zeta\ox A)(B'\ox
\xi)=\varrho(\epi'\ox V)$.  
\end{aussage}

\begin{theorem} \label{thm:bieq}
If idempotent 2-cells in a bicategory $\mathcal K$ split, then the
pseudofunctor $F:\Wfact({\mathcal K})\to
\mathsf{Wdl}({\mathcal K})$ in Paragraph \ref{as:F} is a biequivalence.
\end{theorem}

\begin{proof}
First of all, $F$ is surjective on the objects. In order to see that, take a
weak distributive law $\Psi:A\ox B \to B\ox A$ and evaluate $F$ on the
associated bilinear factorization structure $(\iotaA:A\to B
\ox_\Psi A \leftarrow B:\iotaB, \iota: B\ox_\Psi A \to B \ox A)$ in Theorem
\ref{BPsiA}. The resulting weak distributive law occurs in the top-right path
of  
$$
\xymatrix @C=45pt @R=15pt {
A\ox B \ar[r]^-{\eta\ox A\ox B \ox \eta}\ar[dd]_-\Psi&
B\ox A \ox B \ox A \ar[r]^-{\pi \ox \pi}\ar[dd]^-{B\ox \Psi \ox A}
\ar@{}[rdd]|-{\eqref{wdl12}\eqref{wdl11}}&
(B\ox_\Psi A)\ox (B\ox_\Psi A)\ar[d]^-{\mu_\Psi}\\
&&B\ox_\Psi A\ar[d]^-\iota\\
B\ox A \ar[r]^-{\eta\ox B\ox A \ox \eta}\ar@{=}@/_1pc/[rr]&
B\ox B \ox A \ox A \ar[r]^-{\mu \ox \mu}&
B\ox A.}
$$
\vspace{.2cm}

\noindent
Thus by commutativity of this diagram, it is equal to $\Psi$.

Next we show that $F$ induces an equivalence of the hom categories. The
induced functor of the hom categories is also surjective on the objects. In
order to see that, take a 1-cell $(\xi:A'\ox V\to V \ox A,\zeta:B'\ox V \to V
\ox B)$ in $\mathsf{Wdl}({\mathcal K})$ from the image under $F$ of a 
bilinear factorization structure  $(\iotaA:A\to R\leftarrow B:\iotaB,
\mono:R\to B\ox A)$ to the image of $(\iotaA':A'\to R'\leftarrow B':\iotaB',
\mono':R'\to B'\ox A')$; that is, from the weak distributive law
$\Psi:=\mono\mu_R (\iotaA\ox \iotaB)$ to $\Psi':=\mono'\mu_{R'}(\iotaA'\ox
\iotaB')$. We show that together with    
$$
\varrho:=\big(
\xymatrix{
R'\ox V \ar[r]^-{\mono'\ox V}&
B'\ox A'\ox V\ar[r]^-{B'\ox \xi}&
B'\ox V \ox A \ar[r]^-{\zeta\ox A}&
V\ox B \ox A \ar[r]^-{V\ox \epi}&
V\ox R}\big)
$$
they constitute a 1-cell in $\Wfact({\mathcal K})$. Unitality of
$\varrho$ follows by commutativity of  
$$
\xymatrix{
V\ar@{=}[d]\ar[rrrrr]^-{V\ox \eta}&&&&&
V\ox R 
\ar@{=}
[ddd]\ar@{}[ld]|-{\eqref{nu}}\\
V\ar[rrrr]^-{V\ox \eta\ox \eta}\ar@{=}[d]\ar[rrrd]^-{V\ox \eta\ox \eta}&&&&
V\ox B \ox A \ar[d]_-{V\ox \idemp}
\ar
[ddr]^(.3){V\ox \epi}_(.7){\eqref{kappaiotanu}}
\ar@{}[lld]|-{\eqref{eq:kappaR}}\\
V\ar[r]_-{\eta'\ox \eta'\ox V}\ar[d]_-{\eta'\ox V}&
A'\ox B'\ox V\ar[r]_-{A'\ox \zeta}\ar[d]^-{\Psi'\ox V}
\ar@{}[rrrd]|-{\eqref{eq:wdl_1-cell}}&
A'\ox V \ox B\ar[r]_-{\xi\ox B'}&
V\ox A \ox B \ar[r]_-{V\ox \Psi}&
V\ox B \ox A \ar[d]_-{V\ox \epi}\\
R'\ox V \ar[r]_-{\mono'\ox V}&
B'\ox A'\ox V\ar[r]_-{B'\ox \xi}&
B'\ox V \ox A \ar[r]_-{\zeta\ox A}&
V\ox B \ox A \ar[r]_-{V\ox \epi}&
V\ox R\ar@{=}[r]&
V\ox R.}
$$
The triangular region commutes by the unitality of the monad morphisms $\xi$
and $\zeta$ and the bottom left square commutes by 
$\Psi'(\eta_{A'} \ox \eta_{B'})=\mono'\mu_{R'}(\iotaA'\ox \iotaB') (\eta_{A'}
\ox \eta_{B'})= \mono'\mu_{R'}(\eta_{R'}\ox \eta_{R'})=\mono'\eta_{R'}$. 
Multiplicativity of $\varrho$ is checked on page \pageref{page:rho_multi}. The
regions marked by (m) on page \pageref{page:rho_multi} commute since $\xi$ and 
$\zeta$ are monad morphisms. This proves that $\varrho$ is a monad morphism. 
\hfill

\eject
\thispagestyle{empty}
\begin{landscape}
$$\label{page:rho_multi}
{\color{white} .} 
\hspace{-.7cm}
\scalebox{.92}{\xymatrix{
R'\ox R'\ox V \ar[r]^-{R'\ox \mono'\ox V}\ar[dddddd]_-{\mu'\ox V}
\ar@{}[rdddddd]|-{\eqref{eq:ip_multi}}&
R'\ox B'\ox A'\ox V \ar[r]^-{R'\ox B'\ox \xi}
\ar[d]^-{\mono'\ox B'\ox A'\ox V}& 
R'\ox B'\ox V\ox A \ar[rr]^-{R'\ox \zeta \ox A}&&
R'\ox V \ox B \ox A\ar[rr]^-{R'\ox V \ox \epi}&&
R'\ox V \ox R \ar[d]^-{\mono'\ox V \ox R}\\
&
B'\ox A'\ox B'\ox A'\ox V
\ar[r]^-{\raisebox{5pt}{${}_{B'\ox A'\ox B'\ox \xi}$}}
\ar[d]^-{B'\ox \Psi'\ox A'\ox V}&
B'\ox A'\ox B'\ox V \ox A \ar[rr]^-{B'\ox A'\ox \zeta \ox A}
\ar[d]^-{B'\ox \Psi'\ox V \ox A}
\ar@{}[rrdd]|-{\eqref{eq:wdl_1-cell}}&&
B'\ox A'\ox V \ox B \ox A \ar[rr]^-{B'\ox A'\ox V \ox \epi}
\ar[d]^-{B'\ox \xi \ox B \ox A}&&
B'\ox A'\ox V \ox R \ar[d]^-{B'\ox \xi \ox R}\\
&
B'\ox B'\ox A'\ox A'\ox V 
\ar[r]^-{\raisebox{5pt}{${}_{B'\ox B'\ox A'\ox \xi}$}}
\ar[dd]^-{B'\ox B'\ox \mu'\ox V}
\ar@{}[rdd]|-{(\mathrm{m})}&
B'\ox B'\ox A'\ox V \ox A \ar[d]^-{B'\ox B'\ox \xi \ox A}&&
B'\ox V \ox A \ox B \ox A \ar@{=}[r]\ar[d]^-{B'\ox V \ox \Psi \ox A}&
B'\ox V \ox A \ox B \ox A 
\ar[r]^-{\raisebox{5pt}{${}_{B'\ox V \ox A \ox \epi}$}}
\ar[d]^-{\zeta \ox A \ox B \ox A}&
B'\ox V \ox A \ox R \ar[d]^-{\zeta \ox A \ox R}\\
&&
B'\ox B'\ox V \ox A \ox A 
\ar[r]^-{\raisebox{5pt}{${}_{B'\ox\zeta \ox A\ox A}$}}
\ar[d]^-{B'\ox B'\ox V \ox \mu}&
B'\ox V \ox B \ox A \ox A 
\ar[r]^-{\raisebox{5pt}{${}_{B'\ox V \ox \idemp \ox A}$}}
\ar[d]^-{B'\ox V \ox B \ox \mu}
\ar@{}[rd]|-{\eqref{wdl8}}&
B'\ox V \ox B \ox A \ox A \ar[d]^-{B'\ox V \ox B \ox \mu}&
V \ox B \ox A \ox B \ox A 
\ar[r]^-{\raisebox{5pt}{${}_{V \ox B \ox A \ox \epi}$}}
\ar[d]^-{V \ox B \ox \Psi \ox A}
\ar@{}[rddd]|-{\eqref{eq:ip_multi}}&
V \ox B \ox A \ox R\ar[d]^-{V \ox \epi\ox R}\\
&
B'\ox B'\ox A'\ox V \ar[r]^-{B'\ox B'\ox \xi}\ar[dd]^-{\mu'\ox A'\ox V}& 
B'\ox B'\ox V \ox A \ar[r]^-{B'\ox \zeta\ox A}\ar[dd]^-{\mu'\ox V \ox A}
\ar@{}[rdd]|-{(\mathrm{m})}&
B'\ox V \ox B \ox A \ar[r]^-{B'\ox V \ox \idemp}\ar[d]^-{\zeta \ox B \ox A}&
B'\ox V \ox B \ox A\ar[d]^-{\zeta \ox B \ox A}&
V \ox B \ox B \ox A \ox A\ar[dd]^-{V\ox \mu \ox \mu}&
V\ox R \ox R \ar[dd]^-{V\ox \mu}\\
&&&
V \ox B \ox B \ox A\ar[r]^-{V\ox B \ox \idemp}\ar[d]^-{V\ox \mu \ox A}
\ar@{}[rd]|-{\eqref{wdl8}}&
V \ox B \ox B \ox A \ar[d]^-{V\ox \mu \ox A}&\\
R'\ox V \ar[r]^-{\mono'\ox V}&
B'\ox A'\ox V \ar[r]^-{B'\ox \xi}&
B'\ox V \ox A \ar[r]^-{\zeta\ox A}&
V\ox B \ox A \ar[r]^-{V\ox \idemp}\ar@/^-1.5pc/[rrr]_-{V\ox \pi}
^-{\eqref{kappaiotanu}}&
V\ox B \ox A \ar@{=}[r]&
V\ox B \ox A \ar[r]^-{V\ox \epi}&
V\ox R}}
$$
\end{landscape}

The first diagram in \eqref{eq:wfac_1-cell} commutes by commutativity of 
$$
\xymatrix{
A'\ox V \ar[rrr]^-\xi\ar[rrd]^-{A'\ox V \ox \eta}\ar[rd]_-{A'\ox \eta'\ox V}
\ar[dd]_-{\iotaA'\ox V}&&&
V\ox A \ar[rr]^-{V\ox \iotaA} \ar[d]_-{V\ox A \ox \eta}
\ar@{}[rd]|-{(\ast)}&&
V\ox R \ar@{=}[dd]\ar[ld]_-{V\ox \mono}\\
&A'\ox B'\ox V \ar[r]_-{A'\ox \zeta}\ar[d]^-{\Psi'\ox V}
\ar@{}[rrrrd]|-{\eqref{eq:wdl_1-cell}}&
A'\ox V \ox B\ar[r]_-{\xi\ox B}&
V\ox A \ox B\ar[r]_-{V\ox \Psi}&
V\ox B \ox A \ar[rd]_-{V\ox \epi}&\\
R'\ox V \ar[r]_-{\mono'\ox V}\ar@{}[ru]|-{(\ast)}&
B'\ox A'\ox V\ar[r]_-{B'\ox \xi}&
B'\ox V \ox A \ar[r]_-{\zeta\ox A}&
V\ox B \ox A\ar[rr]_-{V\ox \epi}&&
V\ox R.}
$$
The triangular region
at the top left commutes by the unitality of $\zeta$.
The regions marked by ($\ast$) commute by 
$$
\Psi(A\ox \eta_B)
\stackrel{\eqref{eq:kappaR}}=
\idemp(\eta_B\ox A)
\stackrel{\eqref{kappaiotanu}}=
\iota\pi(\eta_B\ox A)
\stackrel{\eqref{nu}}=
\iota\iotaA.
$$
The second diagram in \eqref{eq:wfac_1-cell} commutes 
by symmetrical considerations.  

The functor induced by $F$ between the hom categories acts on the morphisms as
the identity map, hence it is evidently faithful. It is also full since any
2-cell $\omega:(\xi,\zeta)\to (\xi',\zeta')$ in $\mathsf{Wdl}({\mathcal K})$
is a 2-cell $(\xi,\zeta,\varrho)\to (\xi',\zeta',\varrho')$ in
$\Wfact({\mathcal K})$ by commutativity of 
$$
\xymatrix{
R'\ox V \ar[r]^-{\mono'\ox V}\ar[d]_-{R'\ox \omega}&
B'\ox A'\ox V \ar[r]^-{B'\ox \xi}\ar[d]^-{B'\ox A'\ox \omega}&
B'\ox V \ox A \ar[r]^-{\zeta\ox A}\ar[d]^-{B'\ox \omega \ox A}&
V\ox B\ox A \ar[r]^-{V\ox \epi}\ar[d]^-{\omega \ox B\ox A}&
V\ox R\ar[d]^-{\omega \ox R}\\
R'\ox V' \ar[r]_-{\mono'\ox V'}&
B'\ox A'\ox V' \ar[r]_-{B'\ox \xi'}&
B'\ox V' \ox A \ar[r]_-{\zeta'\ox A}&
V'\ox B\ox A \ar[r]_-{V'\ox \epi}&
V'\ox R.
}
$$
The regions in the middle commute since $\omega$ is a 2-cell in
$\mathsf{Wdl}({\mathcal K})$. 
\end{proof}

\begin{remark}
For an arbitrary bicategory ${\mathcal K}$ -- not necessarily
with split idempotents --, the pseudofunctor $F$ in Paragraph \ref{as:F}
induces a biequivalence  
between $\Wfact({\mathcal K})$ and 
the full subbicategory of $\mathsf{Wdl}({\mathcal K})$ whose 0-cells
are those weak distributive laws $\Psi$ for which the idempotent 2-cell
$\idemp$ splits.  
 It induces in particular a biequivalence between the bicategory of
distributive laws in $\mathcal K$ (as a full subbicategory of
$\mathsf{Wdl}(\mathcal K)$) and the bicategory of strict factorization
structures (as a full subbicategory of $\Wfact({\mathcal K})$),
cf. \cite{RosWoo}.   
\end{remark}

\begin{aussage} {\bf Morphisms with trivial underlying 1-cells.}  
For the algebraists, particularly interesting are those 1-cells in
$\Wfact({\mathcal K})$ and $\mathsf{Wdl}({\mathcal K})$ whose 1-cell
part is trivial -- these are algebra homomorphisms in the usual sense. Such
1-cells form a subcategory of the respective horizontal category. 

In $\Wfact({\mathcal K})$, this means monad morphisms $\varrho:R'\to
R$ which restrict to monad morphisms $\xi:A'\to A$ and $\zeta:B'\to B$,
i.e. for which $\varrho \iotaA'=\iotaA \xi$ and $\varrho \iotaB'=\iotaB
\zeta$. 

The corresponding 1-cells in $\mathsf{Wdl}({\mathcal K})$ are pairs of monad
morphisms $\xi:A'\to A$ and $\zeta:B'\to B$ such that $\idemp (\zeta \ox
\xi)\Psi'= \Psi(\xi\ox \zeta)$.
\end{aussage}
\bigskip
\bigskip

\section{Examples:  Bilinear  factorizations of algebras} 
\label{sec:examples}

The aim of this section is to apply the results in the previous section to the
particular monoidal category -- i.e. one-object bicategory -- of modules over a
commutative ring. (Clearly, in this bicategory idempotent 2-cells split.)
More precisely, we collect here some examples of associative and unital
algebras over a commutative ring $k$ which admit a bilinear 
factorization. Some of these algebras admit a strict factorization as well but
the most interesting ones are those which do not. 

\begin{aussage} \label{as:injective}
{\bf Bilinear factorization via subalgebras.}
The algebra homomorphisms $\alpha:A\to R \leftarrow B:\beta$, occurring in 
a bilinear factorization of an algebra $R$, are not injective in
general. In this paragraph we show however that, for any  bilinear 
factorization structure $(\alpha:A\to R \leftarrow  B:\beta,\iota:R\to B\ox
A)$, there is another  bilinear factorization of $R$ with
injective homomorphisms $\tilde \alpha:\tilde A\to R \leftarrow \tilde
B:\tilde \beta$. We give sufficient and necessary conditions for the latter
factorization to be strict.  

Consider a weak distributive law $\Psi:A\ox B \to B\ox A$, with corresponding
algebra homomorphisms $\alpha:A \to B\ox _\Psi A \leftarrow B:\beta$ obtained
by the corestrictions of $\Psi(A\ox \eta):A\to B \ox A \leftarrow
B: \Psi(\eta\ox B)$, cf. Theorem \ref{BPsiA}. 
Put $\tilde A:=\mathrm{Im}(\alpha)\subseteq B\ox _\Psi A \supseteq
\mathrm{Im}(\beta)=:\tilde B$. Then $\alpha$ factorizes through an
epimorphism $A\twoheadrightarrow \tilde A$ and a monomorphism $\tilde
\alpha:\tilde A \rightarrowtail B\ox _\Psi A$ of algebras. Similarly, $\beta$
factorizes through an epimorphism $B\twoheadrightarrow \tilde B$ and a
monomorphism $\tilde \beta:\tilde B \rightarrowtail B\ox _\Psi A$ of
algebras. 
By Theorem \ref{BPsiA}, 
$\iota : B \ox_{\Psi} A \to B \ox A$ is a $B$-$A$-bimodule section of
$\mu_\Psi(\beta\ox\alpha)$, 
so that 
$$
\tilde \iota :=\big(
\xymatrix{ B \ox_{\Psi} A \ar^-{\iota}[r] & B \ox A \ar@{->>}[r]
  & \tilde B \ox \tilde A} \big)
$$
is a $\tilde B$-$\tilde A$-bimodule section of $\mu_\Psi(\tilde \beta\ox\tilde
\alpha)$. 
Therefore $(\tilde \alpha : \tilde A \rightarrow B \ox_{\Psi} A \leftarrow 
\tilde B:\tilde \beta, \tilde \iota : B \ox_{\Psi} A \rightarrow \tilde B \ox
\tilde A)$  
is a bilinear  factorization of $B \ox_{\Psi} A$ via subalgebras. 

By Theorem \ref{thm:weak_fact} there is a weak distributive law  
$\tilde \Psi : =\tilde \iota\mu_\Psi(\tilde\alpha \ox\tilde \beta)$
such that $\tilde B \ox_{\tilde \Psi} \tilde A$ is isomorphic to $B \ox_{\Psi}
A$.
The weak distributive law $\tilde \Psi$ is a proper distributive law if and
only if both unitality conditions $\tilde \Psi(\tilde A \ox \tilde
\eta)=\tilde\eta \ox \tilde A$ and $\tilde \Psi(\tilde \eta\ox \tilde
B)=\tilde B \ox \tilde \eta$ hold. They amount to commutativity of the
following diagrams. 
\begin{equation}\label{eq:subalg}
\xymatrix @C=30pt @R=10pt {
A \ar[r]^-{A\ox \eta} \ar[dd]_-{\eta\ox A}&
A\ox B \ar[r]^-\Psi&
B\ox A \ar[d]^-{\eta\ox B \ox A \ox \eta}
&
B\ar[r]^-{\eta\ox B}\ar[dd]_-{B\ox \eta}&
A\ox B \ar[r]^-\Psi&
B\ox A \ar[d]^-{\eta\ox B \ox A \ox \eta}\\
&&(A\ox B)^{\ox 2}\ar[d]^-{\Psi^{\ox 2}}
&
&&(A\ox B)^{\ox 2}\ar[d]^-{\Psi^{\ox 2}}\\
B\ox A \ar[r]_-{\eta\ox B \ox A \ox \eta}&
(A\ox B)^{\ox 2}\ar[r]_-{\Psi^{\ox 2}}&
(B\ox A)^{\ox 2}
&
B\ox A \ar[r]_-{\eta\ox B \ox A \ox \eta}&
(A\ox B)^{\ox 2}\ar[r]_-{\Psi^{\ox 2}}&
(B\ox A)^{\ox 2}}
\end{equation}
\end{aussage} 

\begin{aussage}\label{as:e}
{\bf Extension of a distributive law.}
Let $A$ be an (associative and unital) algebra over a commutative ring $k$;
and let $e\in A$ such that $ea=eae$ for all $a\in A$ (so that in particular
$e^2=e$). Then $eA$ is a subalgebra of $A$ though with a different unit
element $e$. 

Assume that $\Phi:eA\ox B \to B \ox eA$ is a distributive law. 
It induces an algebra structure on $B\ox eA$ with unit $1\ox e$ and
multiplication $(b'\ox ea')(b\ox ea)=b'\Phi(ea'\ox b)ea=b'\Phi(ea'\ox
b)a$. The maps 
$$
\alpha:A \to B\ox eA, \qquad a\mapsto 1\ox ea\qquad \textrm{and}\qquad
\beta: B \to B\ox eA, \qquad b\mapsto b\ox e
$$
are clearly algebra homomorphisms inducing the $B$-$A$ bimodule map 
$$
\pi:B\ox A \to B\ox eA, \qquad b \ox a\mapsto b\ox ea.
$$
Since $\pi$ possesses a $B$-$A$ bimodule section $\iota:b\ox ea\mapsto b\ox
ea$, the datum $(\alpha:A\to B\ox eA\leftarrow B:\beta,\iota:B\ox eA\to B\ox
A)$ is a  bilinear factorization structure. Hence by  Theorem
\ref{thm:weak_fact} there is a weak distributive law
$$
\Psi: A\ox B\to B \ox A,\qquad a\ox b\mapsto \Phi(ea\ox b)
$$
such that the weak wreath product algebra $B\ox_\Psi A$ is isomorphic the the
strict wreath product $B\ox_\Phi eA$.

By the above considerations, for any element $e$ of $A$ satisfying $ea=eae$
for all $a\in A$, and for any algebra $B$, there is a weak distributive law
$$
A\ox B \to B \ox A,\qquad a\ox b \mapsto b \ox ea
$$
such that the corresponding weak wreath product is the tensor product algebra
$B \ox eA$ with the factorwise multiplication. 
If $B$ is the trivial $k$-algebra $k$, this gives a weak distributive law 
$$
A \cong A \ox k   \to k \ox A \cong A,\qquad a\mapsto ea
$$
and the corresponding weak wreath product algebra $eA$.
\end{aussage}

\begin{aussage} \label{as:Ore}
{\bf Weak Ore extension.}
Recall (e.g. from \cite{MConRob}) that a {\em quasi-derivation} on an
(associative and unital) algebra $B$  over a commutative ring $k$, consists of
a (unital) algebra homomorphism $\sigma:B\to B$ and a $k$-module map
$\delta:B\to B$ such that 
$$
\delta(bb')=\sigma(b)\delta(b')+\delta(b)b',\qquad \textrm{for}\ b,b'\in B. 
$$
Associated to any quasi-derivation, there is an {\em Ore extension}
$B[X,\sigma,\delta]$ of $B$.
As a $k$-module it is the tensor product of $B$ with the algebra
$k[X]$ of polynomials in a formal variable $X$, equipped with the $B$-$k[X]$
bilinear associative and unital multiplication determined by 
$$
(1\ox X)(b\ox 1)=\sigma(b)\ox X+\delta(b)\ox 1, \qquad \textrm{for}\ b\in B. 
$$
Clearly, the Ore extension is a wreath product of $B$ and $k[X]$ with respect
to a distributive law defined iteratively, see \cite[Example 2.11
(1)]{CIMZ}.  
The following characterization can be found e.g. in \cite[Section
1.2]{MConRob}. An algebra $T$ is an Ore extension of $B$ if and only if the
following hold.  
\begin{itemize}
\item $T$ has a subalgebra isomorphic to $B$;
\item there is an element $X$ of $T$ such that the powers of $X$ are linearly
  independent over $B$ and they span $T$ as a left $B$-module;
\item $XB\subseteq BX+B$.
\end{itemize} 
In what follows, we generalize the notion of a quasi-derivation on $B$ and the 
corresponding construction of Ore extension of $B$. The resulting algebra
$B[X,\sigma,\delta]$ will be a weak wreath product of $B$ with
$k[X]$. However, we also show that it is a proper Ore extension of the image
of $B$ in $B[X,\sigma,\delta]$.    

Let $B$ be an (associative and unital) algebra over a commutative ring $k$,
and let $p$ and $q$ be elements of $B$ such that 
$$
p^2=p,\quad
q^2=0,\quad
pq=q,\quad
qp=0,\quad
\mathrm{and}\quad
pbp=bp,\quad
\textrm{for all}\ b\in B.
$$
Then by a {\em $(p,q)$-quasi-derivation} we mean a couple of $k$-linear maps
$\sigma,\delta:B\to B$ such that the following identities hold for all
$b,b'\in B$:
$$
\begin{array}{lll}
\sigma(bb')=\sigma(b)\sigma(b'),\qquad
&\sigma(1)=\sigma(p)=p,\qquad
&\sigma(q)=0,\\
\delta(bb')=\sigma(b)\delta(b')+\delta(b)b'p,\qquad
&\delta(1)=\delta(p)=q,\qquad
&\delta(q)=0.
\end{array}
$$
So that a $(1,0)$-quasi-derivation coincides with the classical notion of
quasi-derivation recalled above. For example, if $B$ is the algebra of
$2 \times 2$ upper triangle matrices of entries in $k$, we may take 
$$
p:=\left(\begin{array}{rr}1&0\\0&0\end{array}\right),\qquad
q:=\left(\begin{array}{rr}0&1\\0&0\end{array}\right),\qquad
\sigma\left(\begin{array}{rr}a&b\\0&c\end{array}\right):=
\left(\begin{array}{rr}a&0\\0&0\end{array}\right),\qquad 
\delta\left(\begin{array}{rr}a&b\\0&c\end{array}\right):=
\left(\begin{array}{rr}0&a\\0&0\end{array}\right).
$$
In terms of a $(p,q)$-quasi-derivation $(\sigma,\delta)$ on an algebra $B$,
define a $k$-module map $\Psi:B\ox k[X]\to k[X]\ox B$ iteratively as
\begin{eqnarray*}
&&\Psi(1\ox b):=bq\ox X+bp\ox 1,\\
&&\Psi(X\ox b):=\sigma(b)q\ox X^2+(\sigma(b)+\delta(b)q)\ox X + 
\delta(b)p\ox 1,\\
&&\Psi(X^{n+1}\ox b):=\Psi(X^n\ox \sigma(b))X+\Psi(X^n\ox \delta(b)),
\end{eqnarray*}
for $n>1$ and $b\in B$.
By induction in $n$ and $m$, one easily checks the following
properties for all $b,b'\in B$ and $n, m\geq 0$.
\begin{itemize}
\item $\Psi(X^n\ox bp)=\Psi(X^n\ox b)$ and $\Psi(X^n\ox bq)=0$;
\item $b\Psi(X^n\ox 1)=\Psi(1\ox b)X^n$,
\item $(B\ox \mu)(\Psi\ox k[X])(k[X]\ox \Psi)(X^n\ox X^m\ox b)=
\Psi(X^{n+ m}\ox b)$,
\item $(\mu \ox k[X])(B\ox \Psi)(\Psi\ox B)(X^n\ox b\ox b')=\Psi(X^n\ox bb')$.
\end{itemize}
That is, $\Psi$ is a weak distributive law and we may regard the corresponding
weak wreath product $B\ox_\Psi k[X]$ as a {\em weak Ore extension} of $B$.

Note however, that $\Psi$ renders commutative both diagrams in
\eqref{eq:subalg}. Hence $B\ox_\Psi k[X]$ is a strict wreath product of the
subalgebras $\tilde B=\{b(q\ox X + p\ox 1)\ |\ b\in B\}$ 
and $\widetilde{k[X]}$, the latter having the set of powers $\{\Psi(X\ox
1)^n=\Psi(X^n\ox 1)\ |\ n\geq 0\}$ as a $k$-basis. In fact, by the
characterization of Ore extensions recalled above, the weak Ore extension 
$B\ox_\Psi k[X]$ is isomorphic to an Ore extension of $\tilde B$.
\end{aussage}

\begin{aussage}
{\bf Distributive laws over separable Frobenius algebras.}\label{as:sF}
An (associative and unital) algebra $R$ over a commutative ring $k$ is said to
possess a {\em Frobenius} structure if it is a finitely generated and
projective $k$-module and there is an isomorphism of (say) left $R$-modules
from $R$ to $\hat R:=\mathrm{Hom}(R,k)$. A more categorical characterization
is this. Any $k$-algebra $R$ can be regarded as an $R$-$k$ bimodule; that is,
a 1-cell $k\to R$ in the bicategory $\mathsf{Bim}$ of $k$-algebras, bimodules
and bimodule maps. It possesses a right adjoint, the $k$-$R$ bimodule
(i.e. 1-cell $R\to k$) $R$. Whenever $R$ is a finitely generated and
projective $k$-module, the 1-cell $R:k\to R$ possesses also a left adjoint
$\hat R:R\to k$ (with right $R$-action $\varphi \leftharpoonup
r=\varphi(r-)$). A Frobenius structure is then an isomorphism between the
right adjoint $R:R\to k$ and the left adjoint $\hat R:R\to k$. In technical
terms, a Frobenius structure is given by an element $\psi\in \hat R$ (called a 
Frobenius functional) and an element $\sum_i e_i \ox f_i\in R \ox R$ (called a
Frobenius basis) such that $\sum_i \psi(re_i)f_i=r=\sum_ie_i\psi(f_ir)$, for
all $r\in R$. Note that a Frobenius algebra $R$ possesses a canonical
(Frobenius) coalgebra structure with $R$-bilinear comultiplication $r\mapsto
\sum_i re_i \ox f_i=\sum_i e_i \ox f_ir$ and counit $\psi$. 
For more on Frobenius algebras we refer to \cite{Abrams} and
\cite{Street:Frob}. 

A {\em separable} structure on a $k$-algebra $R$ is an $R$-bilinear section of
the multiplication map $R\ox R \to R$. Categorically, this means a
section of the counit of the adjunction $R \dashv R:R \to k$.

Finally, a {\em separable Frobenius} structure on $R$ is a Frobenius structure
$(\psi,\sum_i e_i \ox f_i)$ such that the multiplication $R\ox R \to R$ is
split by the $R$-bilinear comultiplication $r\mapsto \sum_i re_i \ox
f_i=\sum_i e_i \ox f_ir$. In other words, a Frobenius structure $(\psi,\sum_i
e_i \ox f_i)$ such that $\sum_i e_i f_i=1_R$. Categorically, the counit of the
adjunction $R\dashv R:R\to k$ is split by the unit of the adjunction $\hat R
\cong R \dashv R:k\to R$.

For a separable Frobenius algebra $R$, a right $R$-module $M$ and a left
$R$-module $N$, the canonical epimorphism 
$$
\pi:M\ox_k N \to M\ox_R N,\qquad \quad m\ox_k n\mapsto m \ox_R n 
$$
is split by
$$
\iota:M\ox_R N \to M\ox_k N,\qquad m \ox_R n \mapsto \sum_i m.e_i 
\ox_k f_i.n, 
$$
naturally in $M$ and $N$. Thus the image of $\iota$ is isomorphic to $M\ox_R
N$. 

Let $R$ be a $k$-algebra. A monad $A$ on $R$ in $\mathsf{Bim}$ is given by a
$k$-algebra homomorphism $\tilde\eta:R\to A$. (Then $\tilde\eta$ induces an
$R$-bimodule structure on $A$; $\tilde\eta$ serves as the $R$-bilinear unit
morphism; and the $R$-bilinear multiplication $\tilde\mu:A\ox_R A \to A$ is the
projection of the multiplication $\mu:A\ox_k A \to A$ of the $k$-algebra $A$.)
A distributive law in $\mathsf{Bim}$ over $R$ is an $R$-bimodule map
$\Phi:A\ox_R B \to B\ox_R A$ rendering commutative the diagrams 
$$
\xymatrix{
A\ox_R A \ox_R B \ar[r]^-{A\ox_R \Phi}\ar[d]_-{\tilde \mu\ox_R B}&
A\ox_R B\ox_R A\ar[r]^-{\Phi\ox_R A}&
B\ox_R A\ox_R A\ar[d]^-{B\ox_R \tilde \mu}
&
B\ar[r]^-{\tilde\eta\ox_R A}\ar@{=}[d]&
A\ox_R B \ar[d]^-\Phi\\
A\ox_R B \ar[rr]_-\Phi&&
B\ox_R A
&
B\ar[r]_-{B\ox_R \tilde\eta}&
B\ox_R A}
$$
together with their symmetrical counterparts. 
Then $\Phi$ induces on $B\ox_R A$ the structure of a monad in $\mathsf{Bim}$
over $R$ -- that is, an algebra structure 
$(b'\ox_R a')(b\ox_R a)=b'\Phi(a'\ox_R b)a$ and an algebra homomorphism
$\tilde \eta\ox_R \tilde \eta: R \to B\ox_R A$. Moreover, 
$$
\alpha:=\tilde \eta\ox_R A:A\to B\ox_R A \leftarrow B:B\ox_R \tilde
\eta=:\beta
$$
are monad morphisms -- that is, algebra homomorphisms which are
compatible with the homomorphisms $\tilde\eta$.
Composing $\beta\ox_k\alpha:B\ox_k A \to B\ox_R A \ox_k B\ox_R A$ with the
multiplication induced by $\Phi$ on $B\ox_R A$ we re-obtain the canonical 
epimorphism $\pi:B\ox_k A\to B\ox_R A$. 

Whenever $R$ is a separable Frobenius algebra, $\pi$ possesses a $B$-$A$
bimodule section $\iota$ above. That is to say, $(\alpha:A \to B \ox_R
A\leftarrow B:\beta, \iota:B\ox_R A\to B\ox_kA)$ is a  bilinear 
factorization structure. Hence by  Theorem \ref{thm:weak_fact} there is
a weak distributive law of the $k$-algebra $A$ over $B$. Explicitly, it comes
out as 
$$
\xymatrix{
A\ox_k B \ar[r]^-\pi&
A\ox_R B \ar[r]^-\Phi&
B\ox_R A \ar[r]^-\iota&
B\ox_k A
}
$$
with corresponding idempotent 
$$
\xymatrix{
B\ox_k A \ar[r]^-\pi&
B\ox_R A \ar[r]^-\iota&
B\ox_k A\ .
}
$$
Hence the resulting weak wreath product is isomorphic to the algebra $B\ox_R A$
with the multiplication induced by $\Phi$. 
\end{aussage}

\begin{aussage}
{\bf The direct sum of weak distributive laws.}\label{as:dirsum}
Assume that we have a finite collection $\Phi_i:A_i\ox B_i \to B_i \ox
A_i$ of distributive laws between algebras over a commutative ring
$k$. Consider the direct sum algebras $A:=\oplus_i A_i$ (with multiplication
$a_ia'_j=\delta_{i,j} a_ia'_i$ and unit $\sum_i 1_{A_i}$) and $B:=\oplus_i
B_i$. It is straightforward to see that   
\begin{equation}\label{eq:dsum}
A\ox B =\oplus_{i,j} (A_i \ox B_j) \to \oplus_{i,j} (B_j \ox A_i)=B\ox A,
\qquad 
a_i \ox b_j \mapsto \delta_{i,j} \Phi_i(a_i \ox b_i)
\end{equation}
is a weak distributive law. 

We claim that it is of the type in Paragraph \ref{as:sF}. Let $R$ be the
algebra $\oplus_i k$ with minimal orthogonal idempotents $p_i$. Clearly, $R$ is
a separable Frobenius algebra via the Frobenius functional $\psi:R\to k$,
$p_i\mapsto 1$ and the separable Frobenius basis $\sum_i p_i\ox p_i \in R\ox
R$. Thus we conclude that $A\ox_R B$ is isomorphic to $\oplus_i (A_i \ox B_i)$
and $B\ox_R A$ is isomorphic to $\oplus_i (B_i\ox A_i)$. An $R$-distributive
law is given by   
$$
\xymatrix{
A\ox_R B \cong \oplus_i (A_i \ox B_i) \ar[r]^-{\oplus \Phi_i}&
\oplus_i (B_i\ox A_i) \cong B\ox_R A.}
$$
Applying to it the construction in Paragraph \ref{as:sF}, we re-obtain the
weak distributive law \eqref{eq:dsum}.
\end{aussage}

\begin{aussage}\label{as:wba}
{\bf Smash product with a weak bialgebra.}
{\em Weak bialgebras} are generalizations of bialgebras, see \cite{Nill} and 
\cite{BNSz:WHAI}. A weak bialgebra over a commutative ring $k$ is a $k$-module
$H$ carrying both an (associative and unital) $k$-algebra structure
$(\mu,\eta)$ and a (coassociative and counital) $k$-coalgebra structure
$(\delta, \varepsilon)$. The comultiplication is required to be multiplicative
-- equivalently, the multiplication is required to be
comultiplicative. However, multiplicativity of the counit, unitality of the
comultiplication and unitality of the counit are replaced by the weaker axioms  
\begin{eqnarray*}
\varepsilon(ab_1)\varepsilon(b_2c)=&\varepsilon(abc)&=
\varepsilon(ab_2)\varepsilon(b_1c), \quad \textrm{for all}\ a,b,c\in H,\\
(\delta(1)\ox 1)(1\ox \delta(1))=&\delta^2(1)&=
(1\ox \delta(1))(\delta(1)\ox 1),
\end{eqnarray*}
where the usual Sweedler-Heynemann index convention is used for the components
of the comultiplication, with implicit summation understood. In particular, we
write $\delta(1)=1_1\ox 1_2=1_{1'}\ox 1_{2'}$ -- possibly with primed indices
if several copies occur. 

The category of (say) right modules of a weak bialgebra over $k$ is monoidal
though not with the same monoidal structure as the category of
$k$-modules. Indeed, if $M$ and $N$ are right $H$-modules, then there is a
diagonal action $(m\ox n)\leftharpoonup h:=m \leftharpoonup h_1 \ox n
\leftharpoonup h_2$ on the $k$-module tensor product $M\ox N$ but it fails to 
be unital. A unital $H$-module is obtained by taking the $k$-module retract 
$$
M\boxtimes N:=\{m \leftharpoonup 1_1 \ox n \leftharpoonup 1_2\ |\  m\in M,
n\in N\}. 
$$
This defines a monoidal product $\boxtimes$ with monoidal unit 
$$
\{\overline \sqcap(h):=\varepsilon(h1_1)1_2\ |\ h\in H\},
$$ 
with $H$-action
$
\overline \sqcap(h) \leftharpoonup h':=
\overline \sqcap(\overline \sqcap(h)h')=
\varepsilon(h1_{1'})\varepsilon(1_{2'}h'1_1)1_2=
\varepsilon(hh'1_1)1_2=
\overline\sqcap(hh')
$.
With respect to this monoidal structure the forgetful functor from the
category of right $H$-modules to the category of $k$-modules is both monoidal
and opmonoidal (hence preserves algebras and coalgebras) but it is not strict
monoidal. 

A right {\em module algebra} of a weak bialgebra $H$ is a monoid in the
category of right $H$-modules. That is, a $k$-algebra $A$ equipped with an
(associative and unital) right $H$-action such that 
$$
(a \leftharpoonup h_1)(a' \leftharpoonup h_2)=aa' \leftharpoonup h
\qquad \textrm{and}\qquad 
1 \leftharpoonup h = 1 \leftharpoonup \overline \sqcap(h),
$$
for all $a,a'\in A$ and $h\in H$.
For any right $H$-module algebra $A$, there is a weak distributive law
$$
\Psi:A\ox H \to H \ox A,\qquad 
a\ox h \mapsto h_1\ox a \leftharpoonup h_2.
$$
It is multiplicative in $A$ by the $H$-linearity of the multiplication in $A$:  
\begin{eqnarray*}
(H\ox \mu)(\Psi\ox A)(A\ox \Psi)(a'\ox a\ox h)&=&
h_1\ox (a' \leftharpoonup h_2)(a\leftharpoonup h_3)\\
&=& h_1 \ox (a'a) \leftharpoonup h_2=
\Psi(\mu \ox H)(a'\ox a\ox h).
\end{eqnarray*}
Multiplicativity in $H$ follows by multiplicativity of the comultiplication in
$H$: 
\begin{eqnarray*}
(\mu \ox A)(A\ox \Psi)(\Psi\ox H)(a\ox h\ox h')=
h_1h'_1 \ox a \leftharpoonup h_2h'_2=
(hh')_1 \ox a \leftharpoonup (hh')_2=
\Psi(A\ox \mu)(a\ox h \ox h').
\end{eqnarray*}
In order to check the weak unitality condition, note that for all $a\in A$,
$$
1_1 \ox a  \leftharpoonup 1_2 = 
1_1 \ox (1a)  \leftharpoonup 1_2=
1_1 \ox (1 \leftharpoonup 1_2)(a  \leftharpoonup 1_3)=
1_1 \ox (1 \leftharpoonup 1_21_{1'})(a  \leftharpoonup 1_{2'})=
1_1 \ox (1 \leftharpoonup 1_2)a.
$$
Also, for all $h\in H$,
$$
\delta(h1_1)\ox 1_2=
h_11_1\ox h_21_2 \ox 1_3=
h_11_{1'}\ox h_21_{2'}1_1 \ox 1_2=
(h1)_1\ox (h1)_21_1\ox 1_2=
h_1\ox h_21_1 \ox 1_2,
$$
hence $h1_1 \ox 1_2=h_1\varepsilon(h_21_1)\ox 1_2$. With these identities at
hand, 
\begin{eqnarray*}
(H\ox \mu)(\Psi \ox A)(\eta \ox H \ox A)(h\ox a)&=&
h_1\ox (1 \leftharpoonup h_2)a=
h_1\ox (1 \leftharpoonup \overline \sqcap(h_2))a\\
&=& h_1\varepsilon(h_2 1_1)\ox (1\leftharpoonup 1_2)a=
h1_1\ox a \leftharpoonup 1_2\\
&=& (\mu \ox A)(H\ox \Psi)(H\ox A \ox \eta)(h\ox a).
\end{eqnarray*}
The weak wreath product corresponding to $\Psi$ is known as a
{\em weak smash product}, see  \cite{NSzW}.

In the rest of this paragraph we show that the weak distributive law $\Psi$
above is of the kind discussed in Paragraph \ref{as:sF}. 
Let us introduce a further map $\sqcap:H\to H$, $h\mapsto
1_1\varepsilon(h1_2)$. It is easy to see that for any $h,h'\in H$,
\begin{itemize}
\item $\varepsilon \sqcap(h)=\varepsilon(h) =\varepsilon
\overline \sqcap(h)$;
\item $\delta\sqcap(h)=1_1\ox \sqcap(h)1_2$  and $\delta \overline
  \sqcap(h)=1_1 \overline \sqcap(h) \ox 1_2$;
\item $\sqcap(\sqcap(h)h')=\sqcap(hh')=\sqcap ( \overline \sqcap(h)h')$ and 
$\overline \sqcap(\sqcap(h)h')=\overline \sqcap(hh')=
\overline \sqcap (\overline \sqcap(h)h')$; 
\item $\overline \sqcap(h)\sqcap(h')=\sqcap(h') \overline \sqcap(h)$;
\item $\sqcap(h\sqcap(h'))=
\sqcap(\sqcap(h)\sqcap(h'))=
1_1\varepsilon(\sqcap(h)\sqcap(h')1_2)=
\sqcap(h)_1\sqcap(h')_1\varepsilon(\sqcap(h)_2\sqcap(h')_2)=
\sqcap(h)\sqcap(h')$
and symmetrically, $ \overline \sqcap(h \overline \sqcap(h'))=  \overline
\sqcap(h) \overline \sqcap(h')$. 
\end{itemize}
Note that $\overline \sqcap (H)$ possesses a separable Frobenius
structure (cf. \cite{Szlach:Fields}) with Frobenius functional given by the
restriction of $\varepsilon$ and Frobenius basis $\overline \sqcap(1_2)\ox
\overline \sqcap(1_1)=1_2 \ox \overline \sqcap(1_1)$ (where the equality
follows by  
$1_1\ox \overline \sqcap(1_2)=
1_1\ox \varepsilon(1_21_{1'})1_{2'}=
1_1\ox \varepsilon(1_2)1_3=
1_1\ox 1_2$):
$$
1_2\varepsilon(\overline \sqcap(1_1)\overline \sqcap(h))=
\varepsilon(1_1\overline \sqcap(h))1_2=
\varepsilon(\overline\sqcap(h)_1) \overline \sqcap(h)_2=
\overline \sqcap(h)=
\overline \sqcap \sqcap \overline \sqcap(h)=
\varepsilon(\overline \sqcap(h)1_2)\overline \sqcap(1_1).
$$
Hence also the opposite algebra $R:=\overline \sqcap(H)^{op}$ has a separable
Frobenius structure with the same Frobenius functional $\varepsilon$ and
Frobenius basis $\overline \sqcap(1_1) \ox 1_2$. Moreover,
$$
\sqcap( \overline \sqcap(h) \overline \sqcap(h'))=
\sqcap(\sqcap(h) \overline \sqcap(h'))=
\sqcap( \overline \sqcap(h')\sqcap(h))=
\sqcap(h'\sqcap(h))=
\sqcap(h')\sqcap(h)=
\sqcap \overline \sqcap(h')\sqcap \overline \sqcap(h).
$$
That is, the restriction of $\sqcap$ yields an algebra homomorphism $R\to
H$. There is an algebra homomorphism $R\to A$, $r\mapsto
1 \leftharpoonup r$ as well. They induce $R$-actions on $A$ and $H$.
By 
$\overline \sqcap\sqcap=\overline \sqcap$ we conclude that, for all $h\in
H$, $1 \leftharpoonup \overline \sqcap(h)=1\leftharpoonup h=1\leftharpoonup
\sqcap(h)$ and thus
$$
a \leftharpoonup \sqcap(h)=
(a1) \leftharpoonup \sqcap(h)=
(a \leftharpoonup 1_1)(1 \leftharpoonup \sqcap(h)1_2)=
a(1 \leftharpoonup \sqcap(h))=
a(1 \leftharpoonup \overline \sqcap(h)).
$$
Consequently,
$$
\Psi(a(1 \leftharpoonup \overline \sqcap(h')) \ox h)=
\Psi(a \leftharpoonup \sqcap(h')\ox h)=
h_1\ox a \leftharpoonup \sqcap(h')h_2=
(\sqcap(h')h)_1\ox a \leftharpoonup (\sqcap(h')h)_2=
\Psi(a\ox \sqcap(h')h).
$$
This means that $\Psi$ projects to an $R$-distributive law 
$$
A\ox_{R} H \to H \ox_{R} A,\qquad 
a \ox_{R} h \mapsto h_1 \ox_{R} a\leftharpoonup h_2.
$$
Multiplicativity in both arguments is obvious. Unitality follows by 
$$
1_1 \ox_{R} a \leftharpoonup 1_2=
1_1 \ox_{R} a \leftharpoonup \overline \sqcap(1_2)=
1_1 \ox_{R} (1 \leftharpoonup \overline \sqcap(1_2))a=
1_1\sqcap(1_2) \ox_{R} a=
1 \ox_{R} a
$$
and 
$$
h_1 \ox_{R}  1 \leftharpoonup h_2=
h_1 \ox_{R}  1 \leftharpoonup \overline \sqcap(h_2)=
h_1\sqcap(h_2) \ox_{R}  1 =
h \ox_{R}  1 .
$$
Applying the construction in Paragraph \ref{as:sF} to this
$R$-distributive law, it yields a weak distributive law $A\ox H \to H\ox
A$, 
$$
a\ox h \mapsto 
h_1 \sqcap\overline \sqcap(1_1)\ox 
(1\leftharpoonup \overline \sqcap(1_2))(a \leftharpoonup h_2)= 
h_1\sqcap(1_1) \ox a \leftharpoonup h_2\overline \sqcap(1_2).
$$
Since $\sqcap(1_1)\ox 1_2=1_1 \ox 1_2 = 1_1 \ox \overline \sqcap(1_2)$, this
is equal to $\Psi$.
\end{aussage}

\begin{aussage}\label{as:2x2=3}
{$\mathbf {2x2=3}$.}
In this paragraph we present a bilinear factorization of the
algebra $T$ of $2 \times 2$ upper triangle matrices over a field $k$ of
characteristic different from 2, in terms of two copies of the group algebra
$k\mathbb{Z}_2$ of the order 2 cyclic group. So the attitudinizing title
refers to the vector space dimensions: we obtain a 3 dimensional
non-commutative algebra as a weak wreath product of two 2 dimensional
commutative algebras. 

A $k$-linear basis of $T$ is given by 
$$
1:=\left(\begin{array}{rr}1&0\\0&1\end{array}\right),\qquad
a:=\left(\begin{array}{rr}1&1\\0&-1\end{array}\right),\qquad
b:=\left(\begin{array}{rr}-1&0\\0&1\end{array}\right).
$$
These basis elements satisfy $ab=a+b-1$ and $ba=-(a+b+1)$.
Denote the second order generator of the cyclic group $\mathbb{Z}_2$ by $g$
and consider the following algebra homomorphisms. 
$$
\alpha: k\mathbb{Z}_2 \to T, \qquad 
g\mapsto a\qquad \textrm{and}\qquad 
\beta:k\mathbb{Z}_2 \to T, \qquad 
g\mapsto b.
$$
In terms of $\alpha$ and $\beta$, we put
$$
\pi:=\big(\xymatrix{
k\mathbb{Z}_2 \ox k\mathbb{Z}_2\ar[r]^-{\beta\ox \alpha}&
T\ox T \ar[r]^-\mu&
T}\big),
$$
with values 
$$
\pi(1\ox 1)=1,\qquad 
\pi(1\ox g)=a,\qquad
\pi(g\ox 1)=b,\qquad
\pi(g\ox g)=ba=-(a+b+1).
$$
It is straightforward to check that $\pi$ has a section $\iota:T\to
k\mathbb{Z}_2 \ox k\mathbb{Z}_2$
with values
\begin{eqnarray*}
&&\iota(1)=
\frac 1 4 (3\cdot 1\ox 1-1\ox g-g\ox 1-g\ox g),\\
&&\iota(a)=
\frac 1 4 (-1\ox 1+3\cdot 1\ox g-g\ox 1-g\ox g),\\
&&\iota(b)=
\frac 1 4 (-1\ox 1-1\ox g+3\cdot g\ox 1-g\ox g),\\
\end{eqnarray*}
which is a homomorphism of $k\mathbb{Z}_2$-bimodules, with respect to the
action induced by $\beta$ on the first factor and the action induced by
$\alpha$ on the second factor. This shows that $T$ has a bilinear 
factorization in terms of the algebra homomorphisms $\alpha$ and $\beta$.

By Theorem \ref{thm:weak_fact} there is a corresponding weak distributive law
$$
\Psi:=\big(\xymatrix{
k\mathbb{Z}_2 \ox k\mathbb{Z}_2\ar[r]^-{\alpha\ox \beta}&
T\ox T \ar[r]^-\mu&
T\ar[r]^-\iota&
k\mathbb{Z}_2 \ox k\mathbb{Z}_2}\big),
$$
with values
\begin{eqnarray*}
&&\Psi(1\ox 1)=
\frac 1 4(3\cdot 1\ox 1-1\ox g-g\ox 1-g\ox g),\\
&&\Psi(1\ox g)=
\frac 1 4(-1\ox 1-1\ox g + 3\cdot g\ox 1 -g\ox g),\\
&&\Psi(g\ox 1)=
\frac 1 4(-1\ox 1+3\cdot 1\ox g -g\ox 1-g\ox g),\\
&&\Psi(g\ox g)=
\frac 1 4(-5\cdot 1\ox 1+3\cdot 1\ox g +3\cdot g\ox 1 -g\ox g),
\end{eqnarray*}
such that $k\mathbb{Z}_2 \ox_\Psi k\mathbb{Z}_2 \cong T$.
\end{aussage}

\end{document}